\documentclass[11pt]{article}

\usepackage[margin=1in]{geometry}
\usepackage{amsmath, amssymb, amsfonts}
\usepackage{amsthm}
\usepackage{hyperref}
\usepackage{authblk}

\title{Weak convergence rate for the Cox--Ingersoll--Ross process}

\author[1]{Alexandros Pavlis}
\author[1]{Umut \c{C}etin}

\affil[1]{Department of Statistics, London School of Economics and Political Science, 10 Houghton Street, London, WC2A 2AE, UK}

\date{}

\newtheorem{theorem}{Theorem}

\newtheorem{lemma}{Lemma}
\newtheorem{proposition}{Proposition}

\newcommand{\keywords}[1]{%
  \par\vspace{0.05in}
  \noindent\footnotesize{\bfseries Key words.} #1
  \par\vspace{0.1in}\normalsize
}

\newcommand{\msc}[1]{%
  \par\vspace{0.05in}
  \noindent\footnotesize{\bfseries AMS 2020 Subject Classification:} #1
  \par\vspace{0.1in}\normalsize
}

\makeatletter
\gdef\@subjectinfo{}
\newcommand{\subject}[1]{\gdef\@subjectinfo{#1}}
\providecommand{\address}[1]{}% affiliation already provided via \affil
\newcommand{\printfrontmatterinfo}{%
  \begingroup\footnotesize
  \ifx\@subjectinfo\@empty\else
    \par\noindent{\bfseries Subject.} \@subjectinfo\par\fi
  \endgroup}
\makeatother

\begin{document}
	
	%%%% Article title to be placed here
	% \title{...} and the author block are already set in the preamble via
	% authblk; the duplicate declarations below are commented out so authors
	% are not printed twice. The text is preserved for reference.
	%\title{Weak convergence rate for the Cox-Ingersoll-Ross process}
	%
	%\author{%%%% Author details
	%    A. Pavlis$^{1}$, U. \c{C}etin$^{1}$,
	%}
	%%%%%%%%% Insert author address here
	\address{$^{1}$ Department of Statistics, London School of Economics and Political Science, 10 Houghton st, London, WC2A 2AE, UK}
	
	%%%% Subject entries to be placed here %%%%
%	\subject{Stochastic Analysis, Numerical Analysis, Mathematical Finance}

\maketitle
\printfrontmatterinfo

%%%% Abstract text to be placed here %%%%%%%%%%%%
\begin{abstract}
We study the weak convergence rate of a drift-implicit discretisation scheme
for the Cox-Ingersoll-Ross (CIR) process in the regime where the process
remains strictly positive. Specifically, we consider scheme~(4) of
Alfonsi~\cite{A}, which arises naturally from applying a drift-implicit
Euler step to the SDE satisfied by the square root of the CIR process and
admits a unique positive closed-form solution at each time step. Using a
PDE approach combined with a continuous-time SDE representation of the
discretised process, we prove that the weak convergence rate is
$\mathcal{O}(1/N)$ under the Feller condition $2\alpha\geq\theta^2$ and
mild polynomial growth conditions on the payoff function. The proof
requires only elementary techniques and, in particular, avoids the
semi-exact simulation machinery used in earlier work. The methodology
is expected to extend to a broader class of diffusion processes.
\end{abstract}

\msc{93E20, 91G80, 49L20, 91A15}
\keywords{CIR model, Weak convergence, option pricing}

%%%%%%%%%%%%%%%%%%%%%%%%%%%

%%%%%%%%%% Insert the texts which can accomdate on firstpage in the tag "fmtext" %%%%%

\section{Introduction}

The Cox-Ingersoll-Ross (CIR) process is defined by the stochastic
differential equation
\begin{equation}\label{eq:CIR_eq}
  dX_t = k(\mu - X_t)\,dt + \theta\sqrt{X_t}\,dW_t,
  \quad X_0 = x,\quad t\in\mathbb{R}_{+},
\end{equation}
where $W = \{W_t\}_{t\geq 0}$ is a one-dimensional Brownian motion,
$k\geq 0$ is the speed of mean reversion, $\mu\geq 0$ is the long-run
mean, and $\theta > 0$ is the diffusion coefficient. We write
$\alpha := k\mu$ throughout. The process~\eqref{eq:CIR_eq} was
introduced by Cox, Ingersoll and Ross~\cite{CIR} as a model for
short-term interest rates. By Feller's test~\cite{KS}, if $X_0 > 0$ then
\begin{equation}
  \mathbb{P}(X_t > 0,\; t\geq 0) = 1
\end{equation}
provided the Feller condition $2\alpha \geq \theta^2$ holds.

The CIR process has since found widespread use beyond interest rate
modelling, notably as the variance process in the Heston stochastic
volatility model. At the core of most financial applications is the evaluation of
$\mathbb{E}[f(X_T)]$, where $f$ is a payoff function and $T > 0$ is a
fixed horizon. Although the transition density of $X$ is known
explicitly as a scaled non-central chi-squared distribution, the
expectation $\mathbb{E}[f(X_T)]$ rarely admits a closed form for
general~$f$, and one must resort to numerical approximation. Even
though the non-central chi-squared structure of the increments permits
exact simulation in principle~\cite{A}, this is computationally
demanding, and discretisation-based schemes are generally preferred in
practice.
The most natural discretisation is the Euler-Maruyama scheme, which
replaces the continuous dynamics by
\begin{equation}
  \hat{X}_{t_{n+1}} = \hat{X}_{t_n}
  + (\alpha - k\hat{X}_{t_n})\tfrac{T}{N}
  + \theta\sqrt{\hat{X}_{t_n}}\,(W_{t_{n+1}} - W_{t_n}),
\end{equation}
where $t_n = nT/N$. A fundamental difficulty is that this scheme does
not preserve positivity: because the Gaussian increment can be
arbitrarily negative, $\hat{X}_{t_{n+1}}$ can become negative with
positive probability, rendering $\sqrt{\hat{X}_{t_{n+1}}}$ undefined at
the next step. Several remedies have been proposed in the literature,
including taking the absolute value or the positive part of the scheme
\cite{DD,BD}, but the most natural resolution is to use an implicit
scheme, which preserves positivity by construction.

Alfonsi~\cite{A} introduced four implicit discretisation
schemes for the CIR process, each of which guarantees strict positivity
of the iterates. Among these, scheme~(3) of~\cite{A}, also proposed
independently by Brigo and Alfonsi~\cite{BA} in the context of credit
modelling, reads
\begin{equation}\label{eq:BEMscheme3}
	\hat{X}_{t_{n+1}}=	\hat{X}_{t_{n}}+\left(\alpha-{\theta^2\over 2}-k\hat{X}_{t_{n+1}}\right)(t-t_n)+\theta\sqrt{\hat{X}_{t_{n+1}}}(W_t-W_{t_n}).
\end{equation}
Scheme~(4) of~\cite{A}, which is the subject of the present paper,
arises from applying a drift-implicit Euler step to the
SDE satisfied by $\sqrt{X_t}$:
\begin{equation} \label{eq:BEMscheme}
	\sqrt{\hat{X}_t}=\sqrt{\hat{X}_{t_n}}+{\alpha-\theta^2/4\over 2\sqrt{\hat{X}_t}}\,(t-t_n)-{k\over 2}\sqrt{\hat{X}_t}\,(t-t_n)+{\theta\over 2}(W_t-W_{t_n}),\quad t\in(t_n,t_{n+1}].
\end{equation}
Squaring both sides, $\hat{X}_t$ is the unique positive root of the
quadratic equation
\begin{equation} \label{eq:myscheme}
	[2+k(t-t_n)]\hat{X}_t-\left[\theta(W_t-W_{t_n})+2\sqrt{\hat{X}_{t_n}}\right]\sqrt{\hat{X}_t}-(\alpha-\frac{\theta^2}{4})(t-t_n)=0,
\end{equation}
which has a unique positive root whenever $4\alpha>\theta^2$~\cite{A}.
Scheme~\eqref{eq:myscheme} is arguably more natural to implement, as it
does not rely on the special structure of the SDE governing the CIR
process (see p.~357 of~\cite{A}) that is used by Alfonsi to arrive
at~\eqref{eq:BEMscheme3}.

\paragraph{Literature on strong convergence}
Strong convergence of discretisation schemes for the CIR process has
been studied extensively. Deelstra and Delbaen~\cite{DD} and Bossy and
Diop~\cite{BD} established strong convergence results for variants of
implicit schemes. Alfonsi~\cite{A} proved strong convergence of order~$1$
for schemes~(3) and~(4) under the Feller condition. Neuenkirch and
Szpruch~\cite{NS} and Dereich, Neuenkirch and Szpruch~\cite{DNS}
subsequently proved first-order strong convergence for a related class
of drift-implicit methods under optimal conditions.

\paragraph{Literature on weak convergence}
Weak convergence is arguably the more relevant notion for derivative
pricing, since one is interested in the accuracy of
$\mathbb{E}[f(\hat{X}_T)]$ as an approximation to $\mathbb{E}[f(X_T)]$.
For SDEs with globally Lipschitz coefficients, the classical work of
Talay and Tubaro established a weak convergence rate of $\mathcal{O}(1/N)$
for the Euler-Maruyama scheme; see also Milstein~\cite{MGN}. For the CIR
process, the non-Lipschitz nature of the square-root diffusion coefficient
makes these classical results inapplicable. Alfonsi~\cite{A} studied the weak convergence of his schemes under a hypothesis~(HW) on the discretisations, and showed that, conditional on~(HW) holding, a weak convergence rate of $\mathcal{O}(1/N)$ follows. However, (HW) does not hold for the more natural discretisation
scheme~\eqref{eq:myscheme}, while it is valid for~\eqref{eq:BEMscheme3}.

\paragraph{Contribution of the present paper}
The purpose of this article is to prove rigorously that the weak
convergence rate of scheme~\eqref{eq:myscheme}, i.e.\ scheme~(4)
of~\cite{A}, is $\mathcal{O}(1/N)$ under the Feller condition
$2\alpha\geq\theta^2$ and a mild polynomial growth assumption on~$f$.
To the best of our knowledge, this is the first rigorous proof of the
optimal weak convergence rate for this scheme. We highlight three
distinguishing features of our work relative to the existing literature.

First, we work directly with scheme~\eqref{eq:myscheme}, which is arguably the
simplest of Alfonsi's implicit schemes to implement in practice.

Second, the payoff regularity we require is mild: we need only
$f\in\mathcal{C}^{(2)}((0,\infty),\mathbb{R})$ with $|f''(x)|\leq K(1+x^m)$
for some $m\geq 2$, whereas~\cite{A} requires four continuous derivatives
with polynomial growth.

Third, the methodology itself is elementary and general. Following
\c{C}etin and Hok~\cite{CH}, the key step is to derive a continuous-time
It\^{o} SDE satisfied by the interpolated discretised process
$\{\hat{X}_t\}_{t\geq 0}$ (Lemma~\ref{lemma1}). This representation
makes the structure of the discretisation error transparent and reduces
the weak convergence analysis to the estimation of two integral
quantities, $I^n_1$ and $I^n_2$, each of order $\mathcal{O}(1/N^2)$,
via a PDE argument (Proposition~\ref{proposition1}) and moment bounds on
$\hat{X}$ (Lemma~\ref{lemma2}). The moment bounds are obtained via a
comparison theorem argument combined with a time-change. We expect this
methodology to extend to a broader class of diffusions with non-Lipschitz
coefficients, which is left to subsequent work.

The remainder of the paper is organised as follows.
Section~\ref{sec:weak} states and proves the three main auxiliary results
(Lemma~\ref{lemma1}, Lemma~\ref{lemma2} and
Proposition~\ref{proposition1}) and then establishes the main convergence
theorem (Theorem~\ref{T1.1}). Appendix~A collects the detailed
calculations supporting the bounds on $I^n_1$ and $I^n_2$.

\section{Weak Convergence of the CIR model}\label{sec:weak}

Following the work of~\cite{CH}, we start by deriving an appropriate stochastic differential equation for the implicit discretised process $\lbrace\hat{X}_t\rbrace_{t\geq 0}$.

\begin{lemma}\label{lemma1}
	Consider the implicit scheme defined by~\eqref{eq:BEMscheme} with $2\alpha\geq\theta^2$. Then
	\begin{equation}
		\label{eqn:invscheme}
		d\hat{X}_t={\sigma(\hat{X}_t)\over F(\hat{X}_t,t;t_n)}dW_t+{\sigma^2(\hat{X}_t)\over F^2(\hat{X}_t,t;t_n)}\left[g(\hat{X}_t,t;t_n)+{b(\hat{X}_t)\over \sigma^2(\hat{X}_t)}\right]dt,\quad t\in(t_n,t_{n+1}],
	\end{equation}
	where $\sigma:\mathbb{R}_{+}\to\mathbb{R}_{+}$, $\sigma(x)=\theta\sqrt{x}$, $b:\mathbb{R}_{+}\to\mathbb{R}$, $b(x)=\alpha-kx$, and
	\begin{align}
		& F(x,t;t_n)=1+{k\over 2}(t-t_n)+{(4\alpha-\theta^2)(t-t_n)\over 8x}, \nonumber \\
		&g(x,t;t_n)=-{k(t-t_n)\over 2\theta^2}+{1\over 2x}-{2(2+k(t-t_n))\over 4x(2+2k(t-t_n))+(4\alpha-\theta^2)(t-t_n)} \nonumber \\
        &~~~~~~~~~~~~~~+{(4\alpha-\theta^2)^2(t-t_n)\over 32\theta^2 x^2}. \nonumber
	\end{align}
	Moreover, 

\[
\frac{1}{F(x,t;t_n)}\le 1,~
|g(x,t;t_n)|\le K\left(1+\frac{1}{x}+\frac{t-t_n}{x^2}\right),~
\sup_{t\le T,\,N} g(x,t;t_n)\le K\left(\frac{1}{x}+\frac{1}{x^2}\right),
\]
\[
\inf_{x\ge0} g(x,t;t_n)
=
-\frac{k}{2\theta^2}(t-t_n)
\ge
-\frac{kT}{2\theta^2}.
\]
where $K>0$ is a constant depending only on $\alpha$, $\theta$ and $T$.
\end{lemma}
\begin{proof}
	Define $H:\mathbb{R}_{+}\times\mathbb{R}_{+}\to\mathbb{R}$ by $H(x,t;t_n):=(2+k(t-t_n))\sqrt{x}-(\alpha-\theta^2/4){(t-t_n)\over\sqrt{x}}$.

	Equation~\eqref{eq:BEMscheme} implies that
	\begin{equation}
		H(\hat{X}_t,t;\hat{X}_{t_n},t_n)=2\sqrt{\hat{X}_{t_n}}+\theta(W_t-W_{t_n}),
	\end{equation}
	so that
	\begin{equation}
		\hat{X}_t=H^{-1}\!\left(2\sqrt{\hat{X}_{t_n}}+\theta(W_t-W_{t_n})\right),
	\end{equation}
	where $H^{-1}$ denotes the inverse mapping of $H$ in its first argument.

	Note that we cannot directly apply Ito's formula since $H$ is not $\mathcal{C}^{2,1}(\mathbb{R}_{+}\times\mathbb{R}_{+})$. To this end, consider the stopping time $\tau_M:=\inf\lbrace t\geq 0\mid\hat{X}_t\leq {1\over M}\rbrace$ and localise the corresponding stochastic process:
	\begin{align}
		\hat{X}_{t\wedge\tau_M}&=X_0-\int_0^{t\wedge\tau_M}{\partial_s H(\hat{X}_s,s;t_n)\over\partial_xH(\hat{X}_s,s;t_n)}\,ds
		+\int_0^{t\wedge\tau_M}{\theta\over\partial_x H(\hat{X}_s,s;t_n)}\,dW_s \nonumber \\
		&\quad-\int_0^{t\wedge\tau_M}\theta^2{\partial_{xx} H(\hat{X}_s,s;t_n)\over 2(\partial_x H(\hat{X}_s,s;t_n))^2}\,ds.
	\end{align}
	Since the implicit scheme~\eqref{eq:myscheme} has a unique positive solution, $\lim_{M\to\infty}\tau_M=\infty$ $\mathbb{P}$-a.s., so Ito's formula applies on $[0,T]$. A simple rearrangement gives the claimed form
	\begin{equation}
		d\hat{X}_t={\sigma(\hat{X}_t)\over F(\hat{X}_t,t;t_n)}dW_t+{\sigma^2(\hat{X}_t)\over F^2(\hat{X}_t,t;t_n)}\left[g(\hat{X}_t,t;t_n)+{b(\hat{X}_t)\over\sigma^2(\hat{X}_t)}\right]dt,
	\end{equation}
	with
	\begin{align}
		&g(x,t;t_n):=-{1\over 2}{\partial_{xx} H(x,t;t_n)\over\partial_x H(x,t;t_n)}-{1\over\theta^2}\partial_x H(x,t;t_n)\,\partial_t H(x,t;t_n)-{b(x)\over\sigma^2(x)}, \nonumber \\
		&F(x,t;t_n):=\sqrt{x}\,\partial_x H(x,t;t_n).
	\end{align}

	Direct computation yields the expressions for $g$ and $F$ stated in the lemma. The triangle inequality and $4\alpha>\theta^2$ give the claimed bounds on $g$ and $F$.

	Finally, we compute
%	\begin{equation}
%		\partial_x g(x,t;t_n)=\frac{8\left(k(t-t_n)+2\right)^2}{\left(4\alpha(t-t_n)-(t-t_n)\theta^2+4(t-t_n)kx+8x\right)^2}-\frac{(t-t_n)\left(\theta^2-4\alpha\right)^2}{16\theta^2 x^3}-\frac{1}{2x^2}.
%	\end{equation}
    \begin{equation}
\partial_x g(x,t;t_n)
=
\frac{8\bigl(k(t-t_n)+2\bigr)^2}
     {\bigl((t-t_n)(4\alpha-\theta^2+4kx)+8x\bigr)^2}
-\frac{(t-t_n)(\theta^2-4\alpha)^2}
      {16\theta^2x^3}
-\frac{1}{2x^2}.
\end{equation}
	Elementary inequalities give $\partial_x g(x,t;t_n)\leq 0$, so $\inf_{x\geq 0}g(x,t;t_n)=-{k\over 2\theta^2}(t-t_n)\geq -{k\over 2\theta^2}T$.
\end{proof}

The next lemma establishes that certain (inverse) moments of the discretised process are finite. This is crucial for controlling the estimates that govern the weak convergence rate of $\lbrace\hat{X}_t\rbrace_{t\leq T}$ towards the original process $\lbrace X_t\rbrace_{t\geq 0}$. Additionally, it is shown that the discretised process $\lbrace\hat{X}_t\rbrace_{t\leq T}$ does not escape to infinity in the long run -- another property shared with the original process. This further justifies the scheme as a suitable numerical approximation of the CIR process.

\begin{lemma}\label{lemma2}
	Let $\hat{X}_t$ be the process defined by~\eqref{eq:BEMscheme} and $2\alpha\geq\theta^2$. If $0\leq p<{2\alpha\over\theta^2}$ then
	\begin{equation}
		\sup_{t\leq T,N}\mathbb{E}^x\left[\int_{t_n}^t{1\over\hat{X}^{p}_s}{1\over F^2(\hat{X}_s,s;t_n)}ds\right]\leq K(t-t_n)
	\end{equation}
\end{lemma}
for some constant $K>0$. Moreover, $\forall m\geq 0$
\begin{equation}
	\sup_{t\leq T,N}\mathbb{E}^x[\hat{X}^m_t]<\infty,
\end{equation}
and $\mathbb{P}^x(\hat{X}_t<\infty\text{ and }\hat{X}_t>0\text{ for all }t>0)=1$ for all $x>0$.
\begin{proof}
	From~\eqref{eqn:invscheme},
	\begin{equation}
		d\hat{X}_t={\sigma(\hat{X}_t)\over F(\hat{X}_t,t;t_n)}dW_t+{\sigma^2(\hat{X}_t)\over F^2(\hat{X}_t,t;t_n)}\left[g(\hat{X}_t,t;t_n)+{\alpha/\theta^2\over\hat{X}_t}-{k\over\theta^2}\right]dt.
	\end{equation}
	Consider the process $\hat{Y}$ defined by $\hat{Y}_t=\hat{X}_{A^{-1}_t}$, where
	\begin{equation}
		dA_t={1\over F^2(\hat{X}_t,t;t_n)}\,dt,
	\end{equation}
	and $A^{-1}_t:=\lbrace s\geq 0\mid\int_0^s{1\over F^2(\hat{X}_u,u;t_n)}du=t\rbrace$ is the stopping time at which $A$ reaches $t$, with $A_t$ well defined and finite since by~\hyperref[lemma1]{Lemma~\ref*{lemma1}} we have ${1\over F(\hat{X}_t,t;t_n)}\leq 1$, which gives $A_t\leq t$ $\mathbb{P}$-a.s.

	Consequently, the Dambis-Dubins-Schwarz theorem ( e.g.\ Theorem~V.1.6 in~\cite{RY}) yields
	\begin{equation}
		d\hat{Y}_t=\theta\sqrt{\hat{Y}_t}\,dB_t+\theta^2\left[\left(g(\hat{Y}_t,A^{-1}_t;A^{-1}_{t_n})-{k\over\theta^2}\right)\hat{Y}_t+{\alpha\over\theta^2}\right]dt,\quad t\in(A^{-1}_{t_n},A^{-1}_{t_{n+1}}],
	\end{equation}
	where $B$ is a standard Brownian motion adapted to the filtration $(\mathcal{F}_{A^{-1}_t})_{t\geq 0}$.

	Define $Y$ as the CIR process started at $X_0$, i.e.,
	\begin{equation}
		Y_t=X_0+\int_0^t\theta\sqrt{Y_s}\,dB_s+\int_0^t(\alpha-cY_s)\,ds,\quad c=k+{kT\over 2}.
	\end{equation}
	Since $g(x,t;t_n)\geq\inf_x g(x,t;t_n)\geq -{kT\over 2\theta^2}$ by~\hyperref[lemma1]{Lemma~\ref*{lemma1}}, the generalised comparison theorem for stochastic differential equations (see Theorem~2.10 in~\cite{CD}) yields
	\begin{equation}
		\mathbb{P}(\hat{Y}_t\geq Y_t,~t\leq T)=1.
	\end{equation}
	An immediate consequence is
	\begin{equation}\label{auaxeq1}
		\sup_{t\leq T,N}\mathbb{E}^x\left[\int_{t_n}^t{1\over\hat{X}^p_s}{1\over F^2(\hat{X}_s,s;t_n)}ds\right]\leq\sup_{t\leq T}\mathbb{E}^x\left[\int_{A_{t_n}}^{A_t}{1\over Y^p_s}ds\right].
	\end{equation}
	The first claim now follows from the fact that $\sup_{t\leq T}\mathbb{E}[{1\over Y^p_t}]<\infty$ whenever $p<{2\alpha\over\theta^2}$~, as shown by Eq.(3.1) in \cite{DNS}, combined with Fubini's theorem and $A_t\leq t$ $\mathbb{P}$-a.s.

%using ideas from~\cite{HK}

	Next, we prove the last claim. It has already been shown that $\hat{X}_t>0$ $\mathbb{P}$-a.s.\ for all $t\in[0,T]$ when $X_0>0$. Thus, we are left to show that explosions do not occur for the scheme~\eqref{eq:BEMscheme}. By~\hyperref[lemma1]{Lemma~\ref*{lemma1}}, $g$ is bounded above by $\sup_{t\leq T,N}g(x,t;t_n)\leq K\left({1\over x}+{1\over x^2}\right)$, where $K>0$ depends only on $\alpha$, $\theta$ and $T$. Hence the total drift of $\hat{X}$ is bounded above by $K(1+{1\over x})+b(x)$, which is locally Lipschitz.

	To establish the desired finiteness, we introduce the auxiliary process $\{Z_t\}_{t\geq 0}$ starting from $X_0$ and solving
	\begin{equation}
		dZ_t=\left[K\!\left(1+\frac{1}{Z_t}\right)+\alpha-kZ_t\right]dt+\theta\sqrt{Z_t}\,dB_t,
	\end{equation}
	where $B$ is a Brownian motion. Note that $Z$ is not a CIR process; the additional $K(1+1/x)$ term in the drift, with its $1/x$ singularity, makes $Z$ a generalised Bessel-type process. By the comparison theorem~\cite{KS,CD}, since the drift of $\hat{Y}$ is bounded above by that of $Z$, we have
	\begin{equation}
		\mathbb{P}(\hat{Y}_t\leq Z_t,~\forall t\in\mathbb{R}_{+})=1.
	\end{equation}
	It therefore suffices to show that $+\infty$ is inaccessible for $Z$. The scale function of $Z$ is $s(x)=\exp\!\left(-2\int_c^x\frac{d(\xi)}{\beta^2(\xi)}d\xi\right)$, where $d(x)=K(1+1/x)+\alpha-kx$ and $\beta^2(x)=\theta^2 x$. A direct computation gives
	\begin{equation}
		-2\int\frac{d(x)}{\theta^2 x}\,dx=-\frac{2(K+\alpha)}{\theta^2}\ln x+\frac{2K}{\theta^2 x}+\frac{2k}{\theta^2}x+\text{const},
	\end{equation}
	so that
	\begin{equation}
		s(x)\propto x^{-{2(K+\alpha)/\theta^2}}\exp\!\left(\frac{2K}{\theta^2 x}\right)\exp\!\left(\frac{2kx}{\theta^2}\right).
	\end{equation}
	As $x\to+\infty$, the factor $\exp(2kx/\theta^2)$ dominates and $s(x)\to+\infty$, so $+\infty$ is an inaccessible boundary for $Z$~\cite{KS}. As $x\to 0^+$, the factor $\exp(2K/(\theta^2 x))\to+\infty$, so $Z$ is recurrent. This proves the final claim. For the moment bounds, define the stopping time $\zeta_M:=\inf\lbrace t\geq 0\mid\hat{X}_t>M\rbrace$. From~\eqref{eqn:invscheme} we obtain

\begin{align}
\mathbb{E}^x\!\left[\hat{X}_{t\wedge\zeta_M}\right]
&=
\mathbb{E}^x\!\Bigg[
\int_{t_n}^{t\wedge\zeta_M}
\frac{\sigma(\hat{X}_s)}
     {F(\hat{X}_s,s;t_n)}
\,dW_s
\nonumber\\
&\qquad
+
\int_{t_n}^{t\wedge\zeta_M}
\frac{\sigma^2(\hat{X}_s)}
     {F^2(\hat{X}_s,s;t_n)}
\left(
g(\hat{X}_s,s;t_n)
+\frac{b(\hat{X}_s)}
       {\sigma^2(\hat{X}_s)}
\right)
\,ds
\Bigg].
\end{align}
	The sequence $(\zeta_M)$ is localising, reducing the local martingale to a true martingale. Using the first claim of the lemma, we obtain
	\begin{equation}
		\sup_{t\leq T,N}\mathbb{E}^x[\hat{X}_{t\wedge\zeta_M}]\leq K\!\left[1+\int_{t_n}^t\mathbb{E}^x[\hat{X}_{s\wedge\zeta_M}]\,ds\right],
	\end{equation}
	where $K>0$ is a constant. Since explosions do not occur, $\lim_{M\to\infty}\zeta_M=\infty$ $\mathbb{P}$-a.s., so applying Fatou's lemma on the left-hand side, monotone convergence on the right-hand side, and Gr\"{o}nwall's inequality gives
	\begin{equation}
		\sup_{t\leq T,N}\mathbb{E}^x[\hat{X}_t]<\infty.
	\end{equation}

	Suppose now that
	\begin{equation}
		E(m):=\sup_{t\leq T,N}\mathbb{E}^x[\hat{X}^m_t]<\infty,\quad m\geq 1.
	\end{equation}
	By Ito's formula,
	\begin{align}
		\hat{X}^{m+1}_t&=\hat{X}^{m+1}_{t_n}+(m+1)\int_{t_n}^t\hat{X}^{m}_s{\sigma(\hat{X}_s)\over F(\hat{X}_s,s;t_n)}dW_s\nonumber \\
		&\quad+(m+1)\int_{t_n}^t\hat{X}^{m}_s{\sigma^2(\hat{X}_s)\over F^2(\hat{X}_s,s;t_n)}\left(g(\hat{X}_s,s;t_n)-{k\over\theta^2}+{\alpha/\theta^2\over\hat{X}_s}\right)ds\nonumber \\
		&\quad+{m(m+1)\over 2}\int_{t_n}^t\hat{X}^{m-1}_s{\sigma^2(\hat{X}_s)\over F^2(\hat{X}_s,s;t_n)}ds,\quad t\in(t_n,t_{n+1}].
	\end{align}
	Taking expectations for the stopped process and using the bounds on $g$ and $F$ gives
	\begin{equation}\label{eq:lemma2.eq2}
		\mathbb{E}[\hat{X}^{m+1}_{t\wedge\zeta_M}]\leq\mathbb{E}[\hat{X}^{m+1}_{t_n\wedge\zeta_M}]+K\int_{t_n}^{t\wedge\zeta_M}\mathbb{E}[\hat{X}^{m}_s]\,ds+K\int_{t_n}^{t\wedge\zeta_M}\mathbb{E}[\hat{X}^{m-1}_s]\,ds,
	\end{equation}
	where $K>0$ depends only on the parameters and $T$, and the local martingale term vanishes by localisation. Applying Fatou's lemma and monotone convergence on the left- and right-hand sides respectively, using $x^{m-1}\leq 1+x^m$ and~\eqref{eq:lemma2.eq2} recursively, gives
	\begin{equation}
		E(m+1)\leq\hat{X}^{m+1}_0+KE(m)<\infty.
	\end{equation}
	Thus $\sup_{t\leq T,N}\mathbb{E}^x[\hat{X}_t^m]<\infty$ for $m=1$ implies the statement for all $m\geq 1$.
\end{proof}

Next, we follow a standard PDE approach to estimate the error of the numerical approximation. We start with the following result of Alfonsi~\cite{A}.

\begin{proposition}\label{proposition1}
	Let $f\in\mathcal{C}^{(q)}((0,\infty),\mathbb{R})$ and $m\geq q$, $m\in\mathbb{N}$, be such that there exists $K>0$ with
	\begin{equation}
		|f^{(q)}(x)|\leq K(1+x^m)\quad\forall x\geq 0.
	\end{equation}
	Then the function $v:\mathbb{R}_{+}\times[0,T]\to\mathbb{R}$ defined by $v(x,t):=\mathbb{E}^x[f(X_{T-t})]$ satisfies the PDE
	\begin{equation}
		\partial_t v+b(x)\partial_x v+{\sigma(x)^2\over 2}\partial_{xx}v=0,\quad v(T,x)=f(x).
	\end{equation}
	Moreover, the derivatives of $v$ are uniformly bounded, and there exists $K>0$ such that
	\begin{equation}
		\left|\partial^{k}_x\partial^{r}_t v(x,t)\right|\leq K(1+x^{m+q+r})\quad\forall(x,t)\in\mathbb{R}_{+}\times[0,T],\quad k+2r\leq q.
	\end{equation}
\end{proposition}

We are now ready to state and prove the main result.

\begin{theorem} \label{T1.1}
	Let $2\alpha\geq\theta^2$ and $f\in\mathcal{C}^{(2)}((0,\infty),\mathbb{R})$ be such that $|f^{(2)}(x)|\leq K(1+x^m)$ for some $m\geq 2$. Then
	\begin{equation}
		|\mathbb{E}^x[f(X_T)]-\mathbb{E}^x[f(\hat{X}_T)]|\sim\mathcal{O}\!\left({1\over N}\right).
	\end{equation}
\end{theorem}
\begin{proof}
	We use the telescoping identity
	\begin{equation}
		\mathbb{E}^x[f(\hat{X}_T)]-\mathbb{E}^x[f(X_T)]=\mathbb{E}^{x}[v(T,\hat{X}_T)]-v(0,X_0)=\sum_{n=0}^{N-1}\mathbb{E}^x[v(t_{n+1},\hat{X}_{t_{n+1}})-v(t_n,\hat{X}_{t_n})].
	\end{equation}
	By It\^{o}'s formula,
	\begin{equation}
		v(t_{n+1},\hat{X}_{t_{n+1}})-v(t_n,\hat{X}_{t_n})=M_{t_{n+1}}-M_{t_n}+I^n_1+I^n_2,
	\end{equation}
	where
	\begin{align}
		I^n_1&:=\int_{t_n}^{t_{n+1}}\partial_t v(\hat{X}_t,t)\left(1-{1\over F^2(\hat{X}_t,t;t_n)}\right)dt, \\
		I^n_2&:=\int_{t_n}^{t_{n+1}}\partial_x v(\hat{X}_t,t)\,g(\hat{X}_t,t;t_n){\sigma(\hat{X}_t)^2\over F^2(\hat{X}_t,t;t_n)}\,dt, \\
		M_t&:=\int_0^t\partial_x v(\hat{X}_t,t){\sigma(\hat{X}_t)\over F(\hat{X}_t,t;t_n)}\,dW_t,
	\end{align}
	with $M$ a \textit{local} martingale. Before proceeding, we introduce the auxiliary polynomials
	\begin{equation}
		Q(x):=1+x^{m+2}\quad\text{and}\quad P(x):=1+x^{m+3}.
	\end{equation}

	We first show that $M$ is a \textit{true} martingale. Indeed,
	\begin{equation}
		\mathbb{E}^x[\langle M\rangle_u]=\mathbb{E}^x\!\left[\int_0^u\left(\partial_x v(\hat{X}_t,t)\right)^2{\sigma(\hat{X}_t)^2\over F^2(\hat{X}_t,t;t_n)}dt\right]\leq\theta^2\int_0^u\sup_{t\leq T}\mathbb{E}^x\!\left[\left(\partial_x v(\hat{X}_t,t)\right)^2\hat{X}_t\right]dt,
	\end{equation}
	which is finite since $|\partial_x v(x,t)|\leq Q(x)$ and~\hyperref[lemma2]{Lemma~\ref*{lemma2}} bounds the right-hand side. Hence $\mathbb{E}^x[M_{t_{n+1}}-M_{t_n}]=0$. It follows that
	\begin{equation}\label{eq:Th1.1}
		|\mathbb{E}^x[f(X_T)]-\mathbb{E}^x[f(\hat{X}_T)]|=\left|\sum_{n=0}^{N-1}\mathbb{E}^x[I^n_1+I^n_2]\right|\leq\sum_{n=0}^{N-1}\!\left(\mathbb{E}^x[|I^n_1|]+\mathbb{E}^x[|I^n_2|]\right).
	\end{equation}
	Hence, we need to estimate $\mathbb{E}^x[|I^n_1|]$ and $\mathbb{E}^x[|I^n_2|]$.

	The strategy for bounding $I^n_j$, $j=1,2$, is the following. By a further application of It\^{o}'s formula we shall show that $I^n_j\sim\mathcal{O}(1/N^2)$, which gives rise to double integrals of the form
	\begin{equation}
		\mathbb{E}^x\!\left[\int_{t_n}^{t_{n+1}}\!\int_{t_n}^t\mathcal{S}(\hat{X}_s,s;t_n){1\over F^2(\hat{X}_s,s;t_n)}ds\,dt\right].
	\end{equation}
	If $\mathcal{S}(x,t;t_n)$ does not diverge faster than $x^{-2\alpha/\theta^2}$ as $x\to 0^+$, then Fubini's theorem and the first part of~\hyperref[lemma2]{Lemma~\ref*{lemma2}} give the required scaling in $N$. If $\mathcal{S}(x,t;t_n)$ has polynomial growth, the second part of the same lemma combined with Fubini's theorem provides the same bound. Furthermore,~\hyperref[proposition1]{Proposition~\ref*{proposition1}} gives $|\partial_x v(x,t)|\leq Q(x)$ and $|\partial_t v(x,t)|\leq P(x)$ for all $(x,t)\in\mathbb{R}_{+}\times[0,T]$, so we only need to control the terms coming from the structure of $\hat{X}$.

	We now estimate $I^n_1$. Using $\sup_{t\leq T}|\partial_t v(x,t)|\leq P(x)$ and applying It\^{o}'s formula gives
	\begin{align} \label{eqn:I1}
		&\left|\int_{t_n}^{t_{n+1}}\partial_t v(\hat{X}_t,t)\left(1-{1\over F^2(\hat{X}_t,t;t_n)}\right)dt\right|\leq\int_{t_n}^{t_{n+1}}P(\hat{X}_t)\left(1-{1\over F^2(\hat{X}_t,t;t_n)}\right)dt \nonumber\\
		&=\int_{t_n}^{t_{n+1}}\Bigg[\int_{t_n}^t P(\hat{X}_s)\partial_s\Psi(\hat{X}_s,s;t_n)\,ds +\int_{t_n}^t P(\hat{X}_s)\partial_x\Psi(\hat{X}_s,s;t_n)\,d\hat{X}_s\nonumber\\
        &+\int_{t_n}^t\partial_x P(\hat{X}_s)\Psi(\hat{X}_s,s;t_n)\,d\hat{X}_s+\int_{t_n}^t\partial_{xx}P(\hat{X}_s)\Psi(\hat{X}_s,s;t_n){\sigma^2(\hat{X}_s)\over 2F^2(\hat{X}_s,s;t_n)}ds\nonumber\\
        &+\int_{t_n}^t P(\hat{X}_s)\partial_{xx}\Psi(\hat{X}_s,s;t_n){\sigma^2(\hat{X}_s)\over 2F^2(\hat{X}_s,s;t_n)}ds+\int_{t_n}^t\partial_x P(\hat{X}_s)\partial_x\Psi(\hat{X}_s,s;t_n){\sigma^2(\hat{X}_s)\over F^2(\hat{X}_s,s;t_n)}ds\Bigg]dt, \nonumber
	\end{align}
	where we used $F^2(x,t_n;t_n)=1$ and defined $\Psi(x,t;t_n):=1-{1\over F^2(x,t;t_n)}$.

	To illustrate how~\hyperref[lemma2]{Lemma~\ref*{lemma2}} controls these estimates, consider the fourth term on the right-hand side. Since ${1\over F^2(x,t;t_n)}\leq 1$, we have $|\Psi(x,t;t_n)|\leq 1$, and thus
    \begin{align}
&\mathbb{E}^x\!\Bigg[
\int_{t_n}^{t_{n+1}}
\int_{t_n}^{t}
\left|
\partial_{xx}P(\hat X_s)
\Psi(\hat X_s,s;t_n)
\frac{\sigma(\hat X_s)^2}
     {2F^2(\hat X_s,s;t_n)}
\right|
\,ds\,dt
\Bigg]
\nonumber\\
&\le
\frac{\theta^2}{2}
\int_{t_n}^{t_{n+1}}
\int_{t_n}^{t}
\mathbb{E}^x\!\left[
\partial_{xx}P(\hat X_s)\hat X_s
\right]
\,ds\,dt\le
K
\int_{t_n}^{t_{n+1}}
\int_{t_n}^{t}
ds\,dt\sim
\mathcal{O}\!\left(\frac{1}{N^2}\right).
\end{align}
	The remaining terms are handled similarly; full details are given in Appendix~A, where we show using~\hyperref[lemma2]{Lemma~\ref*{lemma2}} that all terms are $\mathcal{O}(1/N^2)$. Hence $\mathbb{E}^x[|I^n_1|]\sim\mathcal{O}(1/N^2)$.

	For the second estimate $I^n_2$, we use $|\partial_x v(x,t)|\leq Q(x)$ from~\hyperref[proposition1]{Proposition~\ref*{proposition1}}. As shown in~\hyperref[lemma1]{Lemma~\ref*{lemma1}}, the last term of $g(x,t;t_n)$ is $\mathcal{O}(t-t_n)$; we separate this term and apply Ito's formula to the remainder:

\begin{align}
&\left|
\int_{t_n}^{t_{n+1}}
\partial_x v(\hat X_t,t)\,
g(\hat X_t,t;t_n)\,
\frac{\sigma(\hat X_t)^2}{F^2(\hat X_t,t;t_n)}
\,dt
\right|\le
\int_{t_n}^{t_{n+1}}
Q(\hat X_t)\,
|g(\hat X_t,t;t_n)|\,
\frac{\sigma(\hat X_t)^2}{F^2(\hat X_t,t;t_n)}
\,dt
\nonumber\\
&\le
\int_{t_n}^{t_{n+1}}
\Bigg[
Q(\hat X_t)(t-t_n)
\left( \frac{k}{2\theta^2}+\frac{(4\alpha-\theta^2)^2}{32\theta^2\hat X^2_t}
%\frac{(4\alpha-\theta^2)/16+\hat X_t}
%     {2\theta^2\hat X_t}
\right)
\frac{\sigma(\hat X_t)^2}{F^2(\hat X_t,t;t_n)}+\int_{t_n}^{t}
Q(\hat X_s)\partial_s\mathcal Z(\hat X_s,s;t_n)\,ds
\nonumber\\
&\qquad
+\int_{t_n}^{t}
\bigl(
\partial_xQ(\hat X_s)\mathcal Z(\hat X_s,s;t_n)
+Q(\hat X_s)\partial_x\mathcal Z(\hat X_s,s;t_n)
\bigr)\,d\hat X_s
\nonumber\\
&\qquad
+\int_{t_n}^{t}
\Bigl(
\partial_{xx}Q(\hat X_s)\mathcal Z(\hat X_s,s;t_n)
+Q(\hat X_s)\partial_{xx}\mathcal Z(\hat X_s,s;t_n)
\nonumber\\
&\hspace{6em}
+2\partial_xQ(\hat X_s)\partial_x\mathcal Z(\hat X_s,s;t_n)
\Bigr)
\frac{\sigma(\hat X_s)^2}
     {2F^2(\hat X_s,s;t_n)}
\,ds
\Bigg]dt .
\end{align}

%	\begin{align}
%		&\left|\int_{t_n}^{t_{n+1}}\partial_x v(\hat{X}_t,t)g(\hat{X}_t,t;t_n){\sigma(\hat{X}_t)^2\over F^2(\hat{X}_t,t;t_n)}dt\right|\leq\int_{t_n}^{t_{n+1}}Q(\hat{X}_t)|g(\hat{X}_t,t;t_n)|{\sigma(\hat{X}_t)^2\over F^2(\hat{X}_t,t;t_n)}dt\nonumber \\
%		&\leq\int_{t_n}^{t_{n+1}}\Bigg[Q(\hat{X}_t)(t-t_n)\!\left({(4\alpha-\theta^2)/16+\hat{X}_t\over 2\theta^2\hat{X}_t}\right){\sigma(\hat{X}_t)^2\over F^2(\hat{X}_t,t;t_n)}+\int_{t_n}^t Q(\hat{X}_s)\partial_s\mathcal{Z}(\hat{X}_s,s;t_n)\,ds \nonumber \\
%		&\quad +\int_{t_n}^t\partial_x Q(\hat{X}_s)\mathcal{Z}(\hat{X}_s,s;t_n)\,d\hat{X}_s+\int_{t_n}^t Q(\hat{X}_s)\partial_x\mathcal{Z}(\hat{X}_s,s;t_n)\,d\hat{X}_s \nonumber \\
%		&\quad +\int_{t_n}^t\partial_{xx}Q(\hat{X}_s)\mathcal{Z}(\hat{X}_s,s;t_n){\sigma^2(\hat{X}_s)\over 2F^2(\hat{X}_s,s;t_n)}ds \nonumber \\
%		&\quad +\int_{t_n}^t Q(\hat{X}_s)\partial_{xx}\mathcal{Z}(\hat{X}_s,s;t_n){\sigma(\hat{X}_s)^2\over 2F^2(\hat{X}_s,s;t_n)}ds \nonumber\\
%		&\quad +\int_{t_n}^t\partial_x% Q(\hat{X}_s)\partial_x\mathcal{Z}(\hat{X}_s,s;t_n)%{\sigma^2(\hat{X}_s)\over %2F^2(\hat{X}_s,s;t_n)}ds\Bigg]dt,
%	\end{align}
where
\begin{align}\label{aux}
\mathcal{Z}(x,t;t_n)
&:=
\tilde g(x,t;t_n)
\frac{\sigma^2(x)}
     {F^2(x,t;t_n)},
\\
\tilde g(x,t;t_n)
&:=
\frac{1}{2x}
-\frac{2(2+k(t-t_n))}
       {2x(2+k(t-t_n))
        +(4\alpha-\theta^2)(t-t_n)}.
\end{align}

%	where $\mathcal{Z}(x,t;t_n):=\tilde{g}(x,t;t_n){\sigma^2(x)\over F^2(x,t;t_n)}$ with $\tilde{g}(x,t;t_n):=\frac{1}{2x}-\frac{2+k(t-t_n)}{2x(2+k(t-t_n))+(4\alpha-\theta^2)(t-t_n)}$.
%    \\
%    \\
    We also used the triangle inequality together with $\tilde{g}(x,t;t_n)\geq 0$ and $g(x,t_n;t_n)=0$.

	The $\mathcal{O}(1/N^2)$ bound for the first term on the right-hand side follows from~\hyperref[lemma2]{Lemma~\ref*{lemma2}}. The remaining terms are shown in Appendix~A to admit finite bounds so $\mathbb{E}^x[|I^n_2|]\sim\mathcal{O}(1/N^2)$.

	Combining the estimates in~\eqref{eq:Th1.1} gives
	\begin{equation}
		|\mathbb{E}^x[f(X_T)]-\mathbb{E}^x[f(\hat{X}_T)]|\sim\mathcal{O}\!\left({1\over N}\right).
	\end{equation}
\end{proof}

\section{Appendix}
We collect here all the bounds used to estimate $I^n_1$ and $I^n_2$ in~\hyperref[T1.1]{Theorem~\ref*{T1.1}}.

\subsection{Calculations for $I^n_1$}
We begin with
\begin{equation}
	F_t(x,t;t_n)={4\alpha-\theta^2\over 8x}+{k\over 2}\Rightarrow{F_t(x,t;t_n)\over F(x,t;t_n)}\leq{4\alpha-\theta^2\over 8x}+{k\over 2}. \nonumber
\end{equation}

Applying~\hyperref[lemma2]{Lemma~\ref*{lemma2}} gives
\begin{equation}
	\mathbb{E}\!\left[\int_{t_n}^{t_{n+1}}\!\int_{t_n}^t\left|P(\hat{X}_s)\partial_s\Psi(\hat{X}_s,s;t_n)\right|ds\,dt\right]\sim\mathcal{O}\!\left({1\over N^2}\right). \nonumber
\end{equation}

Next,
\begin{equation}
	\partial_x\Psi(x,t;t_n)=\partial_x\!\left(1-{1\over F^2(x,t;t_n)}\right)=\frac{128(t-t_n)x\left(\theta^2-4\alpha\right)}{\left[(4\alpha-\theta^2)(t-t_n)+4x(2+k(t-t_n))\right]^3}. \nonumber
\end{equation}

Using the triangle inequality and $4\alpha>\theta^2$, we obtain
\begin{equation}\label{App1eq}
	\left|\partial_x\Psi(x,t;t_n)\right|\leq{K\over t-t_n}~~\text{and}~~\left|\partial_x\Psi(x,t;t_n)\right|\leq{K\over x}, 
\end{equation}
where $K$ depends only on $\alpha$, $\theta$, $T$.

These two inequalities together allow us to control the term
$\partial_x\Psi(x,t;t_n)
\bigl(g(x,t;t_n)+\frac{b(x)}{\sigma^2(x)}\bigr)
\sigma^2(x)$. The first inequality in~\eqref{App1eq} controls the $(t-t_n)/x^2$ term in $g$, while the second controls the remaining terms. Using only the first would give a divergent integral of $1/(t-t_n)$, while using only the second would leave a $1/x^2$ term whose expectation diverges unless we further impose $\alpha>\theta^2$. Using these inequalities, we obtain
\begin{align}
	&\left|\partial_x\Psi(x,t;t_n)\,{(4\alpha-\theta^2)^2(t-t_n)\over 32\theta^2 x^2}\right|\leq{K\over x^2}, \nonumber \\
	&\left|\partial_x\Psi(x,t;t_n)\!\left({b(x)\over\sigma^2(x)}+g(x,t;t_n)-{(4\alpha-\theta^2)^2(t-t_n)\over 32\theta^2 x^2}\right)\right|\leq K\!\left({1\over x}+{1\over x^2}\right). \nonumber
\end{align}
Combining these, we obtain an integrand of the form $K(1+{1\over x}){1\over F^2(x,t;t_n)}$, which ensures that
\begin{equation}
	\mathbb{E}\!\left[\int_{t_n}^{t_{n+1}}\!\int_{t_n}^t\left|P(\hat{X}_s)\partial_x\Psi(\hat{X}_s,s;t_n)\!\left(g(\hat{X}_s,s;t_n)+{b(\hat{X}_s)\over\sigma^2(\hat{X}_s)}\right)\!{\sigma^2(\hat{X}_s)\over F^2(\hat{X}_s,s;t_n)}\right|ds\,dt\right]\sim\mathcal{O}\!\left({1\over N^2}\right). \nonumber
\end{equation}

The next term of concern is
\begin{equation}
	\partial_{xx}\Psi(x,t;t_n)=128(t-t_n)(4\alpha-\theta^2)\,{8x(1+2k(t-t_n))-(t-t_n)(4\alpha-\theta^2)\over(4x(1+2k(t-t_n))+(t-t_n)(4\alpha-\theta^2))^4}. \nonumber
\end{equation}
Elementary algebra gives
\begin{equation}
	\left|\partial_{xx}\Psi(x,t;t_n)\right|\leq{128(t-t_n)(4\alpha-\theta^2)\over(4x(1+2k(t-t_n))+(t-t_n)(4\alpha-\theta^2))^3}\leq{K\over x^2}, \nonumber
\end{equation}
and applying the same arguments yields
\begin{equation}
	\mathbb{E}\!\left[\int_{t_n}^{t_{n+1}}\!\int_{t_n}^t\left|P(\hat{X}_s)\partial_{xx}\Psi(\hat{X}_s,s;t_n){\sigma^2(\hat{X}_s)\over 2F^2(\hat{X}_s,s;t_n)}\right|ds\,dt\right]\sim\mathcal{O}\!\left({1\over N^2}\right). \nonumber
\end{equation}

Finally, we verify that the local martingale terms in~\eqref{eqn:I1},
\begin{align}
	&\int_{t_n}^t P(\hat{X}_s)\partial_x\Psi(\hat{X}_s,s;t_n){\sigma(\hat{X}_s)\over F(\hat{X}_s,s;t_n)}\,dW_s ~&\text{and}~~\int_{t_n}^t\partial_x P(\hat{X}_s)\Psi(\hat{X}_s,s;t_n){\sigma(\hat{X}_s)\over F(\hat{X}_s,s;t_n)}\,dW_s, \nonumber
\end{align}
are true martingales and hence have vanishing expectations. Indeed,
\begin{align}
	&\mathbb{E}^x\!\left[\int_{t_n}^t\!\left(P(\hat{X}_s)\partial_x\Psi(\hat{X}_s,s;t_n){\sigma(\hat{X}_s)\over F(\hat{X}_s,s;t_n)}\right)^2 ds\right]\leq K\mathbb{E}^x\!\left[\int_{t_n}^t{P(\hat{X}_s)/\hat{X}_s\over F^2(\hat{X}_s,s;t_n)}ds\right]<\infty, \nonumber \\
	&\mathbb{E}^x\!\left[\int_{t_n}^t\!\left(\partial_x P(\hat{X}_s)\Psi(\hat{X}_s,s;t_n){\sigma(\hat{X}_s)\over F(\hat{X}_s,s;t_n)}\right)^2 ds\right]\leq K\int_{t_n}^t\mathbb{E}^x[\partial_x P(\hat{X}_s)\hat{X}_s]\,ds<\infty, \nonumber
\end{align}
where for the first quadratic variation we used the second inequality in~\eqref{App1eq}.

Using the same techniques, the remaining terms of~\eqref{eqn:I1} are shown to also be $\mathcal{O}(1/N^2)$ 

\subsection{Calculations for $I^n_2$}
Since $|\mathcal{Z}(x,t;t_n)|=\tilde{g}(x,t;t_n){\sigma^2(x)/F^2(x,t;t_n)}\leq K(1+x)$ and $|g(x,t;t_n)+{b(x)/\sigma^2(x)}|\leq K(1+{1\over x}+{t-t_n\over x^2})$, their product satisfies
\begin{equation}
	\left|\mathcal{Z}(x,t;t_n)\!\left(g(x,t;t_n)+{b(x)\over\sigma^2(x)}\right)\!{\sigma^2(x)\over F^2(x,t;t_n)}\right|\leq K\!\left(1+{1\over x}+x+x^2\right). \nonumber
\end{equation}
Hence
\begin{equation}
	\mathbb{E}\!\left[\int_{t_n}^{t_{n+1}}\!\int_{t_n}^t\left|\partial_x Q(\hat{X}_s)\mathcal{Z}(\hat{X}_s,s;t_n){\sigma^2(\hat{X}_s)\over F^2(\hat{X}_s,s;t_n)}\!\left(g(\hat{X}_s,s;t_n)+{b(\hat{X}_s)\over\sigma^2(\hat{X}_s)}\right)\right|ds\,dt\right]\sim\mathcal{O}\!\left({1\over N^2}\right). \nonumber
\end{equation}

Similarly, a direct computation for $\partial_x Z(x,t;t_n)$ and $\partial_t Z(x,t;t_n)$ using \eqref{aux} gives

\begin{equation}
\left|\partial_t\mathcal{Z}(x,t;t_n)\right|\leq K\!\left({1\over x}+1+x\right){1\over F^2(x,t;t_n)} ~~\text{and}~~ \left|\partial_x\mathcal{Z}(x,t;t_n)\right|\leq K\!\left(1+{1\over x}\right)     \nonumber
\end{equation}

%\begin{align}
%\partial_t\mathcal Z(x,t;t_n)&=
%\frac{\theta^2}{x}
%\frac{1}{F^2(x,t;t_n)}
%\Bigg(
%-\frac{2}{F^2(x,t;t_n)}
%+\frac{(2+(t-t_n)k)(4\alpha-\theta^2+4kx)}
%       {F^2(x,t;t_n)}
%-\frac{32k^2}{\theta^2}x^2
%\Bigg)
%\nonumber\\
%&\quad
%+\frac{\theta^2(4\alpha-\theta^2+4kx)}
%       {8F^3(x,t;t_n)}
%\Bigg(
%\frac{2+(t-t_n)k}
%     {xF(x,t;t_n)}
%+\frac{(t-t_n)k^2}{\theta^2}
%-\frac{1}{x}
%\Bigg).
%\end{align}

%\begin{align}
%	\partial_t\mathcal{Z}(x,t;t_n)=\partial_t\!\left(\tilde{g}(x,t;t_n){\sigma^2(x)\over F^2(x,t;t_n)}\right)&={\theta^2\over x}\,{1\over F^2(x,t;t_n)}\!\left(-{2\over F^2(x,t;t_n)}+{(2+(t-t_n)k)(4\alpha-\theta^2+4kx)\over F^2(x,t;t_n)}-{32k^2\over\theta^2}x^2\right) \nonumber \\
%	&\quad+{\theta^2(4\alpha-\theta^2+4kx)\over 8F^3(x,t;t_n)}\!\left({2+(t-t_n)k\over xF(x,t;t_n)}+{(t-t_n)k^2\over\theta^2}-{1\over x}\right). \nonumber
%\end{align}
%This can be bounded as
%\begin{equation}
%	\left|\partial_t\mathcal{Z}(x,t;t_n)\right|\leq K\!\left({1\over x}+1+x\right){1\over F^2(x,t;t_n)},
%\end{equation}
so that
\begin{equation}
	\mathbb{E}\!\left[\int_{t_n}^{t_{n+1}}\!\int_{t_n}^t Q(\hat{X}_s)\left|\partial_s\mathcal{Z}(\hat{X}_s,s;t_n)\right|ds\,dt\right]\sim\mathcal{O}\!\left({1\over N^2}\right). \nonumber
\end{equation}

%The next term is bounded by
%\begin{equation}
%	\left|\partial_x\mathcal{Z}(x,t;t_n)\!\left(g(x,t;t_n)+{b(x)\over\sigma^2(x)}\right)\!{\sigma^2(x)\over %F^2(x,t;t_n)}\right|\leq K\!\left(1+{1\over x}+x\right), \nonumber
%\end{equation}
\begin{equation}
	\mathbb{E}\!\left[\int_{t_n}^{t_{n+1}}\!\int_{t_n}^t\left|Q(\hat{X}_s)\partial_x\mathcal{Z}(\hat{X}_s,s;t_n)\!\left(g(\hat{X}_s,s;t_n)+{b(\hat{X}_s)\over\sigma^2(\hat{X}_s)}\right)\!{\sigma^2(\hat{X}_s)\over F^2(\hat{X}_s,s;t_n)}\right|ds\,dt\right]\sim \mathcal{O}(\frac{1}{N^2}) \nonumber
\end{equation}

The remaining terms are handled similarly. It remains to verify that the corresponding local martingales are true martingales, hence $I^n_2\sim\mathcal{O}(1/N^2)$. This follows from the bounds
\begin{align}
	&\mathbb{E}^x\!\left[\int_{t_n}^t\!\left(Q(\hat{X}_s)\partial_x\mathcal{Z}(\hat{X}_s,s;t_n){\sigma(\hat{X}_s)\over F(\hat{X}_s,s;t_n)}\right)^2 ds\right]<\infty, \nonumber \\
	&\mathbb{E}^x\!\left[\int_{t_n}^t\!\left(\partial_x Q(\hat{X}_s)\mathcal{Z}(\hat{X}_s,s;t_n){\sigma(\hat{X}_s)\over F(\hat{X}_s,s;t_n)}\right)^2 ds\right]<\infty. \nonumber
\end{align}

%\enlargethispage{20pt}

%%%%%%%%%% Insert bibliography here %%%%%%%%%%%%%%

\end{document}